\documentclass[12pt]{article}
\usepackage{latexsym,color,amsmath,amsthm,amssymb,amscd,amsfonts,mathrsfs,array}
\usepackage{tikz}
\usetikzlibrary{arrows,calc}
\tikzset{
>=stealth',
help lines/.style={dashed, thick}, axis/.style={<->}, important
line/.style={thick}, connection/.style={thick, dotted}, }

\def\R{\mathbb{R}}

\newtheorem{lemma}{Lemma}[section]
\newtheorem{proposition}[lemma]{Proposition}
\newtheorem{remark}[lemma]{Remark}
\newtheorem{theorem}[lemma]{Theorem}
\newtheorem{definition}[lemma]{Definition}

\newtheorem*{remark*}{Remark}

\newtheorem*{def*}{Definition}
\newtheorem*{prop*}{Proposition}

\makeatother
\begin{document}
\title{Bounds for Minkowski Billiard Trajectories in  Convex Bodies}
\author{Shiri Artstein-Avidan, \ Yaron Ostrover}
\date{}
\maketitle
\begin{abstract}
In this paper we use the Ekeland-Hofer-Zehnder symplectic capacity to
provide several bounds and inequalities for the length  of the shortest periodic billiard
 trajectory in a  smooth convex body in ${\mathbb R}^{n}$. Our results
hold both for classical billiards, as well as for the more general case
of Minkowski billiards.
 \end{abstract}

\section{Introduction and results}

The theory of mathematical billiards provides natural models for many
physical problems, and has numerous applications in different fields,
such as dynamical systems, geometric optics, acoustics, and statistical
mechanics to name a few. 
We refer the reader to~\cite{T} and
the references within for an excellent introduction to the subject.

For a smooth convex body $K\subseteq \R^n$, we denote by $\xi(K)$  the
length of the shortest periodic billiard trajectory in $K$
(see Subsection~\ref{Mink-Bill-sec}). In this paper we relate this quantity
to  the Ekeland-Hofer-Zehnder symplectic capacity\footnote{During the preparation of this manuscript
we learned from~\cite{Al,Her} that this relation  may possibly be  obtained via
an approximation scheme developed in~\cite{BG}. Here we  give a direct
proof which avoids this machinery.}.  The precise statement
is given in Theorem~\ref{Main-Theorem} below. Based on this result, we
establish  several estimates for $\xi(K)$ which we now turn to describe.

In fact, our results below apply not only to Euclidean
billiards, but to the more general  case of {\it Minkowski billiards}.
These are the natural generalizations of classical billiards where
the Euclidean structure is replaced by a Minkowski metric.
In particular, from the point of view of geometric optics, Minkowski billiard
trajectories describe the propagation of waves in a homogeneous,
anisotropic medium that contains perfectly reflecting mirrors (see~\cite{GT}).
However, in order to keep the presentation simple, we chose to state our
results in the introduction only for the Euclidean case, and to postpone
the discussion of the analogous results for Minkowski billiards to subsequent sections.

Our first result is the following Brunn-Minkowski type inequality for the length of
the shortest periodic billiard trajectory in a smooth convex body. Here the  Minkowski sum 
of two sets $A, B$ in ${\mathbb R}^{m}$ is defined by $A + B = \{a+b \ ; \ a \in A,b \in B\}$.

\begin{theorem} \label{BM-for-Bill}
For any two  smooth convex bodies $K_1, K_2 \subseteq {\mathbb R}^n$,
one has:
\begin{equation*}
\xi(   K_1 + K_2) \geq \xi(K_1) +  \xi(K_2).
\end{equation*}
Moreover, equality holds if and only if
their exists a closed curve which, up to translation, is a
length-minimizing billiard trajectory in both $K_1$ and $K_2$.
\end{theorem}

In light of the ``classical versus quantum" relation between the
length spectrum 
and the Laplace spectrum via trace formulae and Poisson relations
(see e.g.,~\cite{CdV,GM}),
Theorem~\ref{BM-for-Bill} above can be viewed as a classical counterpart
of a well-known result of Brascamp and Lieb stating that the first
eigenvalue of the Dirichlet Laplace operator on bounded convex
domains satisfies a Brunn-Minkowski type inequality, see~\cite{BL}.

Our next result provides an upper bound for $\xi(K)$ in terms of the
volume ${\rm Vol}(K)$.

\begin{theorem} \label{main-application2}
Let $K \subseteq {\mathbb R}^{n}$ be a smooth convex body. Then,
\begin{equation*} 
\xi (K) \leq C  \sqrt{n} \, {\rm Vol}  (K)^{\frac 1 n},
\end{equation*}
where $C$ is some positive constant independent of the dimension.
\end{theorem}

Note that the case $K=B$, where $B$ is the Euclidean unit ball,
implies that the above inequality is asymptotically sharp, i.e., it is of
the optimal order of magnitude. Another bound for $\xi(K)$,
which follows directly from 
Theorem~\ref{BM-for-Bill} is the following:
\begin{theorem} \label{main-application}
Let $K \subseteq {\mathbb R}^{n}$ be a smooth convex body. Then,
\begin{equation*} 
\xi (K) \leq 2  (n+1) \,  { inrad}(K).
\end{equation*}
\end{theorem}
Here $inrad(K)$ stands for the inradius of $K$ i.e.,
the radius of the largest inscribed ball in $K$.
Numerical computations indicate that up 
to a universal constant (independent of dimension)  the example of the (slightly smoothed) regular
$n$-simplex  shows that  the inequality in Theorem~\ref{main-application} is asymptotically sharp.

Finally, another by-product of Theorem~\ref{Main-Theorem} mentioned
above is the following monotonicity statement, which although seems to be known
to experts in the field, has not been
addressed in the literature to the best of our knowledge.
\begin{proposition} \label{prop-about-monotonicity}
Let  $K_1 \subseteq  K_2 \subseteq {\mathbb R}^n$ be two smooth
convex bodies. Then,
\begin{equation*}
\xi(K_1) \leq \xi(K_2).
\end{equation*}
\end{proposition}

For comparison, let us now briefly describe some  known results
regarding the length of the shortest periodic billiard trajectory. We start
with the following theorem by Ghomi which, for a convex body $K$,  gives a sharp lower
bound for $\xi(K)$.

\begin{theorem}[Ghomi~\cite{Gh}]\label{thm-ghomi} Let
$K \subseteq {\mathbb R}^n$ be a convex body, Then,
\begin{equation*}   \xi(K) \geq 4 \, { inrad}(K). \end{equation*}
Moreover, equality holds if and only if ${width}(K ) = 2 \, {
inrad}(K )$, and then the shortest  billiard trajectory is
2-periodic.
\end{theorem}
Here, 
$width(K)$ is the
thickness of the narrowest slab which contains $K$, and a trajectory
is 2-periodic if it has exactly two bouncing points. We remark that
the definition of a billiard trajectory in~\cite{Gh} is
slightly less general than the one used here.
However, both definitions coincide for the shortest periodic billiard trajectory on a smooth strictly convex body.

An upper  bound for the length of the shortest periodic  billiard
trajectory inside an arbitrary bounded domain $\Omega \subseteq {\mathbb R}^n$ with smooth
boundary  in terms of the volume ${\rm Vol}(\Omega)$ was given by Viterbo~\cite{V1},
 who proved that:
\begin{theorem}[Viterbo~\cite{V1}]  \label{Thm-Viterbo}
Let $\Omega$ be a bounded domain in ${\mathbb
R}^n$  with smooth boundary. Then there exists a billiard trajectory in
$\Omega$ of length $l$ with
\begin{equation} \label{Viterbo-result-on-vol}  l \leq C_n {\rm Vol} (\Omega)^{\frac 1 n } ,
\end{equation}
where $C_n$ is a positive  constant which depends  on the dimension  satisfying
$C_n = O(2^n)$.
\end{theorem}

Theorem~\ref{Thm-Viterbo} was improved by Irie~\cite{Ir}, who used
a certain symplectic capacity,  defined via the theory of symplectic
homology that was introduced in~\cite{V3}, to show:
\begin{theorem} [Irie~\cite{Ir}]
Let $\Omega$ be a bounded domain in ${\mathbb
R}^n$  with smooth boundary. Then there exists a billiard trajectory
 in $\Omega$ of length $l$ with
\begin{equation} \label{Irie-result-on-inrad} l \leq C_n  {\, inrad}(\Omega),
\end{equation}
where $C_n$ is a positive  constant which depends   on the dimension $n$.
\end{theorem}
We remark that the constant $C_n$ given in~\cite{Ir} is not explicit and 
no estimate for it is given. 
Moreover, note that on the class of convex bodies in $\R^n$,  results of
the type~$(\ref{Viterbo-result-on-vol})$ and~$(\ref{Irie-result-on-inrad})$ 
i.e., with {\em some} dimension-dependent  
constant $C_n$, are easily obtained by a standard
compactness argument. Hence, for this class, the main task is to give a good bound on
the constant $C_n$ as a function of the dimension. This is the
objective  in our Theorems~\ref{main-application2} and~\ref{main-application} above.
Finally, while this note was under preparation, we learned of the paper~\cite{AlM} where the following theorem, in which a dimension independent bound for $l$ is given in terms of the diameter, was proved:
\begin{theorem}[Albers-Mazzucchelli~\cite{AlM}]
Let $\Omega$ be a bounded domain with smooth boundary in ${\mathbb
R}^n$. Then there exists a billiard trajectory in $\Omega$ of length
$l$ with
\begin{equation*} l \leq C  \, diam(\Omega), \end{equation*}
where $C > 0 $ is a  constant independent of $n$, and $diam(\Omega) =
\inf \{ |v|  ; \, (v + \Omega) \cap \Omega = \emptyset \}$.
\end{theorem}

\noindent {\bf Notations:}
By a convex body we shall mean a compact convex
 set  with non-empty interior.  The class of convex bodies in ${\mathbb R}^{n}$
is denoted by  $\widetilde {\mathcal K}^{n}$, and the subclass of convex bodies
with smooth boundary by ${\mathcal K}^{n}$. A body is said to be strictly convex if it has strictly positive Gauss curvature at every point of its boundary. 
Given $\Sigma \in \widetilde  {\mathcal K}^n$, we denote by
$h_\Sigma : {\mathbb R}^n \rightarrow {\mathbb R}$ its support
function given by $h_\Sigma(u) = \sup \{ \langle x, u \rangle \ ; x \in \Sigma \}$,
where $\langle \cdot , \cdot \rangle$ stands for the standard inner product
in ${\mathbb R}^n$. We denote by  $g_\Sigma  : {\mathbb R}^n \rightarrow {\mathbb R}$
the gauge function $g_\Sigma(x) = \inf \{ r  \ | \  x \in  r\Sigma\}$ associated with $\Sigma$.
Note that when $\Sigma$ is centrally symmetric i.e., $\Sigma = -\Sigma$,
one has that $g_{\Sigma}(x)$ is a norm, which we denote by $\| x \|_{\Sigma}$. 
Furthermore, when $0\in int(\Sigma)$, one has that $h_\Sigma = g_{\Sigma^{\circ}}$, where  
$\Sigma^{\circ} = \{y \in {\mathbb R}^{n} \ | \
\langle x,y \rangle \leq 1, \ {\rm  for \ every \ } x \in \Sigma \}$ is the polar body of $\Sigma$.
For a smooth function $F\colon {\mathbb R}^n \rightarrow  {\mathbb R}$, we
write $\nabla F$ and $\nabla^2 F$ for its gradient and Hessian respectively.
Finally, given $\gamma \colon S^{1} \rightarrow {\mathbb R}^{2n}$,  we shall state claims holding for $\gamma(t)$ omitting the phrase ``for all $t\in S^1$'' so as not to needlessly complicate the text.

\noindent {\bf Structure of Paper:} The paper is organized as follows.
In Section~\ref{Perliminaries}, after providing the relevant background
from symplectic geometry and the theory of Finsler billiards, we
state our main results regarding the relation between the length of the shortest periodic
billiard trajectory and the Ekeland-Hofer-Zehnder capacity (Theorem~\ref{Main-Theorem}).
In Section~\ref{Sec-proof-of-main-results} we use this relation to prove our main results.
In Section~\ref{sec:gpdich} we prove a dichotomy between gliding and proper billiard trajectories. In the Appendix  we prove certain technical claims that
were used throughout the text.

\noindent {\bf Acknowledgement:} 
The second named author thanks Peter Albers, David Hermann, and Sergei Tabachnikov for stimulating discussions on billiards and dynamics.
We would also like to thank the anonymous referee for useful comments, and for pointing out a flaw  in an earlier version of this paper.
The first named author was partially  supported by ISF grant No. 247/11.
The second named author was partially supported by a Reintegration Grant SSGHD-268274 within the 7th European community framework programme, and by the ISF grant No. 1057/10. Both authors were partially supported by BSF grant number 2006079. 
\section{Preliminaries} \label{Perliminaries}

Before we turn to prove our main results, we provide some
relevant background from symplectic geometry, and the theory of
Finsler billiards.

\subsection{Symplectic capacities}

Consider the $2n$-dimensional Euclidean space ${\mathbb R}^{2n}  =
{\mathbb R}^n_q \times {\mathbb R}^n_p$  equipped with the linear
coordinates $(q_1,\ldots, q_n,p_1,\ldots,p_n)$,  the standard
symplectic structure $\omega_{st} = \sum dq \wedge dp$, and the
standard inner product $g_{st} = \langle \cdot, \cdot \rangle$.
Note that under the usual identification between ${\mathbb R}^{2n}$
and ${\mathbb C}^n$, these two structures are the real and the imaginary
 parts of the standard Hermitian inner product in ${\mathbb C}^n$ and
 $\omega_{st}(u, Jv) = \langle u, v \rangle$,  where $u,v \in {\mathbb R}^{2n}$,
 and  $J : {\mathbb R}^{2n}
\rightarrow {\mathbb R}^{2n}$  is the  standard complex structure on
${\mathbb R}^{2n}$ given by $J(q,p) = (-p,q)$.
Recall that a symplectomorphism of ${\mathbb R}^{2n}$ is a
diffeomorphism which preserves the symplectic structure i.e., $\psi
\in {\rm Diff}({\mathbb R}^{2n})$ such that $\psi^* \omega_{st} =
\omega_{st}$.

Symplectic capacities are symplectic invariants which, roughly
speaking, measure the symplectic size of subsets of ${\mathbb
R}^{2n}$. More precisely,

\begin{definition} \label{Def-sym-cap}
A symplectic capacity on $({\mathbb R}^{2n},\omega_{st})$ associates
to each  subset $U \subset {\mathbb R}^{2n}$ a number $c(U) \in
[0,\infty]$ such that the following three properties hold:

\noindent (P1) $c(U) \leq c(V)$ for $U \subseteq V$ (monotonicity)

\noindent (P2) $c \big (\psi(U) \big )= |\alpha| \, c(U)$ for  $\psi
\in {\rm Diff} ( {\mathbb R}^{2n} )$ such that $\psi^*\omega_{st} =
\alpha \, \omega_{st}$ (conformality)

\noindent (P3) $c \big (B^{2n}(r) \big ) = c \big (B^2(r) \times
{\mathbb C}^{n-1} \big ) = \pi r^2$ (nontriviality and
normalization).
\end{definition}

Here, $B^{2k}(r)$ stands for the open $2k$-dimensional ball of radius $r$.
Note that the third property disqualifies any volume-related
invariant, while the first two imply that for $U, V \subset {\mathbb
R}^{2n}$, a necessary condition for the existence of a
symplectomorphism $\psi $ 
 with $\psi(U) = V$, is $c(U) =c(V)$ for any symplectic capacity $c$.

It is a priori unclear that symplectic capacities exist. The first
examples were provided by Gromov~\cite{G} 
using pseudo-holomorphic curves techniques. 
Since Gromov's work, several other symplectic capacities were
constructed, such as the Hofer-Zehnder capacity~\cite{HZ,HZ1}, the
Ekeland-Hofer capacities~\cite{EH,EH1}, the displacement
energy~\cite{H1,LaMc}, spectral capacities~\cite{FGS,HZ,Oh,V2}, and
Hutchings'  embedded contact homology capacities~\cite{Hu1}, to name
a few. We refer the reader to~\cite{CHLS} for a detailed survey on
the theory of symplectic capacities.

\subsection{The Ekeland-Hofer-Zehnder capacity} \label{the-EHZ-cap-subsection}

Two important examples of symplectic capacities, which arose from the
study of periodic solutions of Hamiltonian systems, and play a fundamental
rule in this paper, are the Ekeland-Hofer capacity $c_{EH}$ introduced
 in~\cite{EH}, and the Hofer-Zehnder capacity $c_{HZ}$ introduced in~\cite{HZ1}.
As we shall see below, on the class of smooth convex bodies in ${\mathbb R}^{2n}$,
these two capacities coincide, and are given by the minimal action over all
closed characteristics on the boundary of the corresponding convex domain.
Hence, in what follows, we omit the general definition of these two capacities, and
give an equivalent definition which coincides with the standard
ones on the class of smooth convex bodies. This is done in
Theorem~\ref{Cap_on_covex_sets} below (cf. Proposition~\ref{proposition-EHZ-capacity}).

Recall that the restriction of the symplectic form
$\omega_{st}$ to a smooth closed connected hypersurface $\Sigma
\subset {\mathbb R}^{2n}$ defines a 1-dimensional subbundle
${ker}(\omega_{st}|\Sigma)$ whose integral curves comprise the
characteristic foliation of $\Sigma$. In other words, a {\it closed
characteristic} $\gamma$ on $\partial \Sigma$ is an embedded circle
in  $\partial \Sigma$ tangent to the characteristic line bundle
\begin{equation*} {\mathfrak S}_{\Sigma} = \{(x,\xi) \in T
\partial \Sigma \ | \ \omega_{st}(\xi,\eta) = 0 \ {\rm for \ all} \ \eta \in T_x
\partial \Sigma \}. \end{equation*}
The classical geometric problem of finding a closed characteristic 
has the following dynamical interpretation. If  the boundary
$\partial \Sigma$ is represented as a regular energy surface $\{x
\in {\mathbb  R}^{2n} \ | \ H(x) = {\rm const} \}$ of a smooth
Hamiltonian function $H : {\mathbb  R}^{2n} \rightarrow {\mathbb
R}$, then the restriction to $\partial \Sigma$ of the Hamiltonian
vector field $X_H$, defined by $i_{X_H} \omega_{st} = -dH$, is a section
of ${\mathfrak S}_{\Sigma}$. Thus, the image of the periodic
solutions of the classical Hamiltonian equation $\dot x= X_H(x) = J\nabla H(x)$ on
$\partial \Sigma$ are precisely the closed characteristics of
$\partial \Sigma$. In particular, the closed characteristics do not depend (up to
parametrization) on the choice of the Hamiltonian function. Indeed, if $\Sigma$ can
be represented as a regular level set of some other function $F \colon
{\mathbb  R}^{2n} \rightarrow {\mathbb R}$, then $X_{H} = d X_{F}$ on $\partial \Sigma$ for
some $d \neq 0$, and the corresponding Hamiltonian equations have the
same solutions up to  parametrization.

The action $A(\gamma)$ of a closed curve $\gamma$, which is the
enclosed symplectic area, is defined by $A(\gamma) = \int_{\gamma} \lambda$,
where  $\lambda =pdq$ is the Liouville 1-form whose differential is
$d \lambda = \omega_{st}$. Parametrizing $\gamma$ by $\gamma(t)$,
with $0 \leq t \leq 2 \pi$ and $\gamma(0)=\gamma(2\pi)$, the action
 becomes $A(\gamma) = {\frac 1 2} \int_{0}^{2\pi} \langle J \gamma(t),
 \dot \gamma(t) \rangle dt.$
The action spectrum of $\Sigma$ is defined as:
\begin{equation*}  {\cal L}(\Sigma) = \left \{ \, | \, {A}({\gamma}) \,  | \, ;
\, \gamma \ {\rm closed \ characteristic \ on} \ \partial \Sigma
\right \}.\end{equation*}

The following theorem, which as explained above serves here also as the definition 
of the Ekeland-Hofer and Hofer-Zehnder capacities for the class of smooth convex bodies, 
is a combination of results from~\cite{EH} and~\cite{HZ}, and is based on the dual 
action principle introduced by Clarke in~\cite{Cl} (cf.~\cite{R,W}).
\begin{theorem} \label{Cap_on_covex_sets} Let $\Sigma\subseteq {\mathbb
R}^{2n}$ be a convex bounded domain with smooth boundary $\partial
\Sigma$. Then there exists at least one closed characteristic $\gamma^*
\subset \partial \Sigma$ satisfying
\begin{equation*}  c_{EH}(\Sigma) = c_{HZ}(\Sigma) = { A}(\gamma^*) =  \min {\cal
L}( \Sigma). \end{equation*}
\end{theorem}
Such a closed characteristic which minimizes the action (note
that there might be more than one), is called throughout the
text a capacity carrier for $\Sigma$. In addition, we refer to the
coinciding Ekeland-Hofer and Hofer-Zehnder capacities on the class ${\mathcal K}^{2n}$,
of convex bounded  domains in ${\mathbb R}^{2n}$ with non-empty interior and smooth boundary,
as the Ekeland-Hofer-Zehnder capacity, and denote it from here onwards by $c_{_{EHZ}}$.

Another characterization of the Ekeland-Hofer-Zehnder capacity,
which shall be useful for us later on,
is given in the following proposition, the proof of which can be found in~\cite{HZ,MZ}.
Let
$W^{1,2}(S^1,{\mathbb R}^{2n})$ be the Hilbert space of
absolutely continuous $2 \pi$-periodic functions whose  derivative
exist almost everywhere and belongs to $L_2(S^1,{\mathbb R}^{2n})$, and set
\begin{equation*} 
 {\cal E} = \Bigl \{ z \in W^{1,2}(S^1,{\mathbb
R}^{2n}) \ ; \  \int_0^{2 \pi} z(t) dt = 0, \  {\frac 1 2} \int_0^{2
\pi} \langle Jz(t), \dot{z}(t) \rangle dt = 1 \Bigr \}. \end{equation*}

\begin{proposition} \label{proposition-EHZ-capacity}
For  $\Sigma \in {\mathcal K}^{2n}$ 
one has:
\begin{equation*}\label{ep}
 c_{_{\it EHZ}}(\Sigma)=  \min_{z \in
{\cal E}} {\frac {\pi} {2}}
  \int_0^{2 \pi} h_\Sigma^{2}(\dot{z}(t))dt. \end{equation*}
\end{proposition}
 
\subsection{Singular convex energy levels}

Although the above definition of the Ekeland-Hofer-Zehnder capacity $c_{_{\it EHZ}}$ was
given only in the class of convex bodies with smooth boundary, it can be naturally extended
to the class $\widetilde {\mathcal K}^{2n}$ of convex sets in
${\mathbb R}^{2n}$ with non-empty interior (see e.g.~\cite{Ku,Ku1}). Indeed, this is an immediate consequence
 of the following  lemma, the proof of which is straightforward and thus omitted.

\begin{lemma} \label{tech-lemma-about-extension}
Let $f:  {\cal K}^{2n} \rightarrow [0,\infty]$ be 
homogeneous, and monotone with respect to inclusion. Then there is a
unique continuous\footnote{With respect to the Hausdorff topology on
the class of convex bodies.} extension $\widetilde f:  \widetilde
{\cal K}^{2n} \rightarrow [0,\infty]$, given by
\begin{equation*} \label{def-of-extension} \widetilde f(\Sigma) =
\inf \{ f(\Sigma') \ | \ \Sigma' \in  {\cal K}^{2n} , \ {\rm such \ that } \ \Sigma \subset
\Sigma' \},
\end{equation*}
which is monotone,
homogeneous, and coincides with $f$ on the class ${\cal K}^{2n} $.
\end{lemma}

In what follows, we  denote by $ \widetilde c_{_{\it EHZ}}$ the
unique extension of the Ekeland-Hofer-Zehnder capacity to the class
$\widetilde {\cal K}^{2n} $ provided by
Lemma~\ref{tech-lemma-about-extension}. Note that $ \widetilde
c_{_{\it EHZ}}$ is a symplectic capacity on this class i.e., it satisfies
the assumptions of Definition~\ref{Def-sym-cap}.
Moreover, Proposition~\ref{proposition-EHZ-capacity} above 
extends to the class $\widetilde {\cal K}^{2n} $ verbatim (cf.~\cite{Ku,Ku1}) i.e., 
\begin{proposition} \label{proposition-EHZ-capacity-non-smooth}
For any convex body $\Sigma \in \widetilde {\cal K}^{2n}$ 
one has:
\begin{equation*} 
\widetilde  c_{_{\it EHZ}}(\Sigma)=  \min_{z \in
{\cal E}}
  {\frac {\pi} {2 }} \int_0^{2 \pi} h_{ \Sigma}^{2}(\dot{z}(t))dt. \end{equation*}
\end{proposition}
The proof of this proposition is given in the Appendix of the paper.
Furthermore, as in the smooth case,  $\widetilde c_{_{\it EHZ}}(\Sigma)$ is given by the minimal action among
 (generalized) closed characteristics on the boundary of $\Sigma$.
To state this precisely,
we first introduce some notations. We denote the unit outward normal cone of $\Sigma \in {\cal K}^{2n}$ at $x \in {\mathbb R}^{2n}$ by
$$ N_{\Sigma}(x) = \{ u \in {\mathbb R}^{2n} \ | \ \langle u, x-y \rangle \geq 0, \ {\rm for \ every \ } y \in \Sigma \}. $$
This is a set-valued vector field, which for a smooth point 
$x \in \partial \Sigma$ 
equals ${\mathbb R}_{+}\hat{n}$, where $\hat{n}$ is the usual normalized outward normal.
Next, let
$\partial H$ denote the subdifferential of
a function $H \colon {\mathbb R}^{2n} \rightarrow {\mathbb R}$ i.e.,
$$ \partial H (x) := \{ u  \in {\mathbb R}^{2n} \, | \, H(y) \geq H(x)
+ \langle u, y- x \rangle, \ {\rm for \ all \ } y \in {\mathbb R}^{2n} \}.  $$
It is well known (see e.g.,~\cite{Sch}) that for a convex function $H$,
the subdifferential is a non-empty convex subset of ${\mathbb R}^{2n}$,
and that $H$ is differentiable at $x$ if and only if $\partial H(x)$ is a singleton.
Each element of $\partial H(x)$ is called a subgradient of $H$ at $x$.
\begin{definition} A generalized  closed characteristic on a
(possibly singular)  convex hypersurface $\Sigma \subset {\mathbb R}^{2n}$
is the image of a piecewise smooth loop $\gamma : S^1 \rightarrow \partial \Sigma$  which satisfies that for every $t \in S^1$ one has:
\begin{equation} \label{def-of-gen-clos-char}
\dot \gamma^{\pm}(t) \in J N_{\Sigma}(\gamma(t)). 
\end{equation}
 \end{definition}

Here, the loop $\gamma \colon S^1 \rightarrow {\mathbb R}^{2n}$ is
said to be piecewise-smooth if it is continuous, has left and right derivatives at all points, and there is some
zero-measure set $F \subset S^1$ such that $\gamma$ is smooth on $S^1
\setminus F$.  Moreover, it is not difficult to check that condtion~$(\ref{def-of-gen-clos-char})$ above is equivalent to the fact  
 that there are vectors $u^{\pm}(t) \in  J \partial (g_{\Sigma})^{\alpha}(\gamma(t)) $ for some (any) $\alpha \geq 1$ such that 
$\dot \gamma^{\pm} (t) \, || u^{\pm}(t)$.

It is worthy to note that in fact the condition $\gamma(S^1) \subset \partial \Sigma$ 
can be weakened to, say, $\gamma(0) \in \partial \Sigma$, since the assumption on $\dot{\gamma}$, together with the fact that $\gamma$ is periodic, and that $\Sigma$ is convex, already implies that $\gamma(t)\in \partial \Sigma$ for all $t$, see \cite{Ku1}. 
Moreover, we remark that in contrast with characteristics on hypersurfaces with smooth boundary,
where (the image of) two characteristics are either disjoint or coincide, two different
generalized characteristics can  intersect. 

We are now in a position to describe the extension of Proposition~\ref{proposition-EHZ-capacity}
to the class $\widetilde {\mathcal K}^{2n}$. We remark that the proposition below is stated in~\cite{Ku1} and proved
in~\cite{Ku}. However, as of yet, we did not manage to obtain a copy of~\cite{Ku} and provide an independent proof in the Appendix.

\begin{proposition}\label{prop-about-the-1-1-corrospendence}
Let $\Sigma \in \widetilde {\cal K}^{2n}$. Then,
\begin{equation} \label{ANS1}
\widetilde  c_{_{\it EHZ}}(\Sigma)=  \min_{z \in
{\cal E}}
  {\frac {\pi} {2 }} \int_0^{2 \pi} h_{ \Sigma}^{2}(\dot{z}(t))dt  = \min {\widetilde { \cal L }}(\Sigma), 
\end{equation}
where  $ {\widetilde {\cal L}}(\Sigma)$ is the generalized action spectrum of $\Sigma$ defined by:
\begin{equation*}  {\widetilde {\cal L}}(\Sigma) = \left \{ \, | \, {A}({\gamma}) \,  | \, ;
\, \gamma \ {\rm is \ a \ generalized \ closed \ characteristic \ on} \ \partial \Sigma
\right \}.\end{equation*}
\end{proposition}

The proofs  of Propositions~\ref{proposition-EHZ-capacity-non-smooth} and~\ref{prop-about-the-1-1-corrospendence} go
along the same lines as the analogous proofs in the smooth case.  
For completeness,ת we include the details in the Appendix of this paper.
Finally, we  conclude this subsection with the following Brunn-Minkowski type
 inequality that was proved in~\cite{AAO1} for smooth convex bodies.
\begin{theorem} \label{BM-non-smooth-case} For any $\Sigma_1, \Sigma_2 \in \widetilde {\cal K}^{2n}$, one has
\begin{equation} \label{eq-BM-for-HEZ} \widetilde c_{_{EHZ}}(\Sigma_1 + \Sigma_2)^{\frac 1 2}
\geq \widetilde  c_{_{EHZ}}(\Sigma_1)^{\frac 1 2} +
\widetilde c_{_{EHZ}}(\Sigma_2)^{\frac 1 2}. \end{equation}
Moreover, equality holds if and only if $\Sigma_1$ and $\Sigma_2$ have a pair of
capacity carriers which coincide up to translation and dilation.
\end{theorem}
\begin{proof}[{\bf Proof of Theorem~\ref{BM-non-smooth-case}}]
The statement of the theorem  was proved in~\cite{AAO1}  for convex bodies
with smooth boundary. To show \eqref{eq-BM-for-HEZ}, it is not hard to check that $ \widetilde c_{_{EHZ}}$
is continuous with respect to the Hausdorff metric on
$\widetilde {\mathcal K}^{2n}$ (see~\cite{MS}, Exercise 12.7), which immediately implies the inequality  for any  two convex bodies
 $\Sigma_1, \Sigma_2 \in \widetilde {\cal K}^{2n}$. For the characterization of the equality case, one must follow the proof of Proposition \ref{proposition-EHZ-capacity-non-smooth}  for smooth convex bodies from~\cite{AAO1} with a power $p=1$ in the integrand instead of $p=2$, which 
carries over verbatim, including the equality case, to the class $\widetilde {\cal K}^{2n}$.
\end{proof}

\subsection{Minkowski billiards} \label{Mink-Bill-sec}

The general study of billiard dynamics in Finsler and Minkowski
geometries was initiated in~\cite{GT} (see also~\cite{T}). We remark
that from the point of view of geometric optics, Minkowski billiard
trajectories describe the propagation of waves in a homogeneous,
anisotropic medium that contains perfectly reflecting mirrors (see~\cite{GT}).
Below we focus on the special case of Minkowski billiards in a smooth bounded convex
body $K \subset {\mathbb R}^n$. Roughly speaking, we equip $K$
with a Finsler metric given by a certain norm $\| \cdot \|$, and
consider billiards in $K$ with respect to the geometry induced by $\|
\cdot \|$.

More precisely, let $K \subset {\mathbb R}^n_q$, and $T \subset {\mathbb R}^n_p$ be two convex bodies with smooth boundary, and consider the 
unit cotangent bundle
\begin{equation*}
U_T^*K := K \times T = \{ (q,p) \, | \, q \in K, \ {\rm and} \
g_T(p)  \leq 1 \} \subset T^* {\mathbb R}^n_q  = {\mathbb R}^{n}_q
\times  {\mathbb R}^{n}_p \end{equation*} 
Note that when $T$ is centrally symmetric i.e., $T=-T$, one has $g_T(x) = \|x\|_T$.
We remark that although $K$ and $T$ are smooth convex bodies, their product
 $K\times T $ is a smooth manifold with
corners.

Motivated by the classical correspondence between closed
geodesics in a Riemannian manifold  and closed
characteristics of its unit cotangent bundle, we  now relate the generalized closed
characteristics on $K \times T$ with
certain billiard trajectories, which we call $(K,T)$-billiard trajectories.
These are closed billiard trajectories in $K$ when the bouncing rule
is determined by the geometry induced from the body $T$.

As was explained in the previous subsection, after a standard re-parametrization
argument any closed characteristic $\Gamma$ on a smooth convex hypersurface
$\partial \Sigma \subset {\mathbb R}^{2n}$ is the image of a loop
$\gamma : [0,2\pi] \rightarrow {\mathbb R^{2n}}$  where $\dot \gamma =
d J \nabla g_{\Sigma}(\gamma)$, for some constant $d$.  For $K \times T$ one has
$g_{K \times T}(q,p) = \max \{ g_K(q),g_T(p) \}$. This leads naturally to the following:

\begin{definition} \label{def-of-periodic-traj}
A closed $(K,T)$-billiard trajectory is the image of a piecewise smooth
map $\gamma \colon S^1 \rightarrow \partial (K \times T) $
such that for every  $t \notin {\mathcal B}_{\gamma}:= \{ t
\in S^1 \, | \, \gamma(t) \in \partial K \times \partial T \}$ one has:
\begin{equation*}
\dot \gamma(t) = d \, {\mathfrak X}(\gamma(t)),  \end{equation*} for some positive
 constant $d$, and a vector field ${\mathfrak X}$  given by:
\begin{equation*}
{\mathfrak X}(q,p) = \left\{
\begin{array}{ll}
(-\nabla g_T(p) ,0), &  (q,p) \in int(K) \times \partial T,\\
(0,\nabla g_K(q)), & (q,p) \in \partial K \times int(T).
\end{array} \right.
\end{equation*}
Moreover, for any $t \in {\mathcal B}_{\gamma}$, the left and right
derivatives of $\gamma(t)$ exists, and
\begin{equation} \label{eq-the-cone}
\dot \gamma^{\pm}(t) \in \{   \alpha (-\nabla g_T(p) ,0) + \beta
(0,\nabla g_K(q))    \ | \ \alpha,\beta \geq 0, \ {\rm and} \ (\alpha, \beta) \neq (0,0) \}.
\end{equation}
\end{definition}

\begin{remark}
{\rm Although  in Definition~\ref{def-of-periodic-traj} there is a natural symmetry between the  bodies $K$ and $T$, 
in what follows the body $K$ shall play the role of the billiard table, while  $T$ induces the geometry that governs the billiard dynamics in $K$.
It will be useful to introduce the following notation: Let $\pi_q \colon {\mathbb R}^{2n} \rightarrow {\mathbb R}^n_q$ denote the projection to the configuration space. For every 
$(K,T)$-billiard trajectory $\gamma$, the curve $\pi_q(\gamma)$
shall be called a $T$-billiard trajectory in $K$.
} \end{remark}

\begin{definition}[{\bf Trajectories classification}]
A closed $(K,T)$-billiard trajectory $\gamma$ is said to be {\it proper}
if the set ${\mathcal B}_{\gamma}$ is finite, i.e., $\gamma$ is a  a
broken bicharacteristic that enters, and instantly exits, the boundary
$\partial K \times \partial T$ at the reflection points.
In the case where ${\mathcal B}_{\gamma} = S^1$, i.e., $\gamma$ is travelling
solely along the boundary $\partial K \times \partial T$,
 we say that $\gamma$ is a gliding trajectory.
\end{definition}

\begin{figure} 
\begin{center}
\begin{tikzpicture}[scale=1]

 \draw[important line][rotate=30] (0,0) ellipse (75pt and 40pt);

 \path coordinate (w1) at (2.1,4*0.75) coordinate (q0) at
 (-4.5*0.5,-2*0.2) coordinate (q1) at (1.6,4*0.44) coordinate (q2) at
 (1.74,+0.46) coordinate (w2) at (2.6,-1.1) coordinate (w3) at (-3.2,-0.11) coordinate (K) at
 (0,0.15);

\draw[red] [important line] (q0) -- (q1); \draw[->] (q1) -- (w1);
\draw[red] [important line] (q1) -- (q2); 
\draw[->] (q0) -- (w3);

\filldraw [black]
  (w3) circle (0pt) node[above] {{\footnotesize $w_2=\nabla \|q_2\|_K$}}
    (w1) circle (0pt) node[right] {{\footnotesize $w_1=\nabla \|q_1\|_K$}}
    (w2) circle (0pt) 
     (q0) circle (2pt) node[below right] {{\footnotesize $q_2$}}
      (q1) circle (2pt) node[above right=0.5pt] {{\footnotesize $q_1$}}
       (q2) circle (2pt) node[right=0.5pt] {{\footnotesize $q_0$}}
        (K) circle (0pt) node[right=0.5pt] {$K$};

       \begin{scope}[xshift=7cm]

 \draw[important line][rounded corners=10pt][rotate=10] (1.8,0) --
 (0.8,1.8)-- (-0.8,1.8)--  (-1.8,0)--  (-0.8,-1.8) -- (0.8,-1.8) --
 cycle;

 \path coordinate (p1) at (0.5,-4*0.433) coordinate (np0) at
 (2.3,2*0.9) coordinate (p0) at (1.2,2*0.53) coordinate (np1) at
 (0.7,-3) coordinate (p2) at (-1.2,2*0.66) coordinate (D) at
 (0.3,0.22);

 \draw[<-] (np0) node[right] {{\footnotesize $v_1=\nabla \|p_1
 \|_T$}} -- (p0);
 \draw[<-] (np1) node[right] {{\footnotesize $v_0=\nabla \|p_0
 \|_T$}} -- (p1);

 \draw[blue][important line] (p0) -- (p1);
 \draw[blue][important  line] (p0) -- (p2);

  \filldraw [black]
       (p1) circle (2pt) node[below right] {{\footnotesize $p_0$}}
         (p2) circle (2pt) node[left] {{\footnotesize $p_2$}}
         (p0) circle (2pt) node[above=2pt] {{\footnotesize $p_1$}}
          (D) circle (0pt) node[left] {$T$};
 \end{scope}
 \end{tikzpicture}

 \caption{A proper $(K,T)$-Billiard trajectory.} 
 \end{center}
 \end{figure}
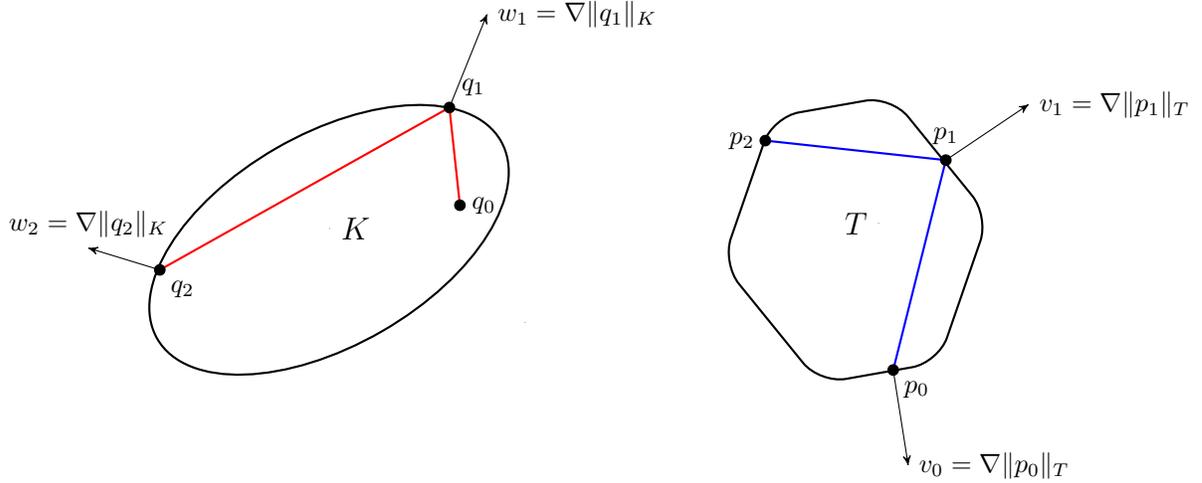

For a proper billiard trajectory, when we follow the flow of the vector field
${\mathfrak X}$, we move in $K \times \partial T$ from $(q_0,p_0)$ to
$(q_1,p_0)$ following the opposite of the outer normal to
$\partial T$ at $p_0$. When we hit the boundary $\partial K$ at the
point $q_1$, the vector field is changing, and we start to move in
$\partial K \times T$ from $(q_1,p_0)$ to $(q_1,p_1)$ following the outer
 normal to $\partial K$ at the point $q_1$.  Next, we move from
$(q_1,p_1)$ to $(q_2,p_1)$ following the opposite of the normal to
$\partial T$ at $p_1$, and so forth  (see
Figure $1$). Note that the reflection law described above is a
natural variation of the classical one (i.e., equal impact and
reflection angles) when the Euclidean structure on ${\mathbb R}^n_q$
is replaced by the metric induced by the norm $\| \cdot \|_{T}$.
Moreover, it is not hard to check that when $T=B$ is the Euclidean
unit ball, the billiard bouncing rule described above is the
standard one.

Note moreover that the action of a $(K,T)$-billiard trajectory $\gamma$ between  two
consecutive bouncing points, say $(q_0,p_0)$ at $t=0$ and $(q_{1},p_0)$ at
$t=\tau_0$, is given by
\begin{equation*}
A_{(q_0,p_0) \rightarrow (q_1,p_0)  }(\gamma) = \int_0^{\tau_0} p(t) \dot q(t) dt = p_0 (q_1-q_0),
\end{equation*}
where $g_T(p_0)=1$, and $q_1-q_0 =d\tau_0 \nabla g_T(p_0)$. Note that this is
also the maximum of the function $p \rightarrow p(q_1-q_0)$ on $g_T^{-1}(1)$,
which by definition equals  $h_T(q_1 -q_0)$. When moving on $\partial K \times T$,
the function $t \rightarrow q(t)$ is constant, and thus the action is zero.
Hence we conclude that the action of a proper $(K,T)$-billiard trajectory 
$\gamma$ with $m$ bouncing points is given by
\begin{equation} \label{action-is-length} A(\gamma) = \sum_{j=1}^m h_T(q_{j+1}-q_{j}),
\end{equation} \label{eq-action-as-length}
where $q_i= \pi_q(\gamma(t_i))$ are the projections to ${\mathbb R}^n_q$ of the
bouncing points $\{ \gamma(t_i)  \ | \ t_i \in {\cal B}_{\gamma} \}$, and $q_{m+1} = q_1$.

The next proposition shows that   
every $(K,T)$-billiard trajectory is either a proper or a gliding trajectory.
This result seems to be  known in the Euclidean case (see e.g.~\cite{GM}). Set
 \begin{equation*} {\cal A} = \{ (q,p) \in \partial K \times \partial T \ |
 \  \nabla g_T(p) \perp \nabla g_K(q) \}. \end{equation*}
It is not hard to check that ${\cal A} \subset \partial K \times \partial T$ is
smooth, and $dim({\cal A})=2n-3$.

\begin{proposition} \label{dichotomy-prop}
 Let $K \subset {\mathbb R}^n_q$ and $T \subset {\mathbb R}^n_p$ be
two smooth strictly convex bodies. 
Then,  every closed $(K,T)$-billiard trajectory is either  a proper trajectory , or a gliding trajectory.
Moreover,  if $\gamma$ is a $(K,T)$-gliding trajectory, then for every $t \in S^1$
one has $\gamma(t) \in {\cal A}$, and
\begin{equation*} \dot \gamma(t) =   (\dot \gamma_q(t), \dot \gamma_p(t))=
 (- \alpha(t) \, \nabla g_{T}(\gamma_p(t)), \beta(t)  \nabla g_{K}(\gamma_q(t))), \end{equation*}
where $\alpha $ and $ \beta $ are two smooth  {\em positive}   functions which satisfy 
\begin{equation}  \label{condition-on-alpha-and-beta}
{\frac {\alpha(t) }   {\beta(t)  }   }  = {\frac  {  \langle \nabla^2 g_{T} (\gamma_p(t)) \nabla g_{K}(\gamma_q(t)), \nabla g_{K}(\gamma_q(t)) \rangle       }    {    \langle \nabla^2 g_{K} (\gamma_q(t)) \nabla g_{T}(\gamma_p(t)), \nabla g_{T}(\gamma_p(t)) \rangle    }   }.
\end{equation}

\end{proposition}

The proof of Proposition~\ref{dichotomy-prop}  is given in the Appendix of this paper. 
We are now in a position to state the theorem on which the results stated in the introduction rely. 

\begin{theorem} \label{Main-Theorem}
Let $K \subset {\mathbb R}^n_q$ and $T \subset {\mathbb R}^n_p$ be
two smooth strictly convex bodies.  Then
there exists at least one  periodic $(K,T)$-billiard trajectory $\gamma^* $
such that
\begin{equation} \label{eq-c-hat-as-action} \widetilde c_{_{\it EHZ}} (K \times T) =  A(\gamma^*)
 = \min \{A(\gamma) \, ; \, \gamma \ {\rm is \ a } \ (K,T) \ {\rm billiard \ trajectory}  \}.  \end{equation}
\end{theorem}

Note that from Theorem~\ref{Main-Theorem} and equation~$(\ref{action-is-length})$ above it follows that $\widetilde c_{_{\it EHZ}} (K \times T)$ is the length, with respect to the
support function $h_T$,  of the shortest $T$-billiard trajectory in $K$. 
We denote this number also by $\xi_T(K) := \widetilde c_{_{\it EHZ}} (K \times T)$. 
In particular, in the case where
 $T=B$ is the Euclidean unit ball, $\xi(K) :=  \widetilde c_{_{\it EHZ}} (K \times B)$ is the length of the shortest periodic (Euclidean) billiard trajectory in $K$. Note that a gliding trajectory is also considered as a billiard in this paper, in contrast to several other settings in the literature on billiards.

\begin{proof}[Proof of Theorem \ref{Main-Theorem}]
Consider the product $K \times T$ of two smooth strictly convex bodies $K \subset {\mathbb R}^n_q$ and $T \subset {\mathbb R}^n_p$.
 Note that  generalized closed characteristics on the boundary
$\partial (K \times T)$ are exactly  $(K,T)$-billiard trajectories and vice versa.
Thus, equation~$(\ref{eq-c-hat-as-action})$ follows  immediately from the combination of
Propositions~\ref{proposition-EHZ-capacity-non-smooth} and~\ref{prop-about-the-1-1-corrospendence}.
\end{proof}

\section{Proof of the main results} \label{Sec-proof-of-main-results}
\subsection{Brunn-Minkowski for billiards}

Let us first prove the following Minkowski-billiard version of the inequality in Theorem \ref{BM-for-Bill}. Let $K_1, K_2 \subseteq {\mathbb R}_q^n$ and   $T \subseteq {\mathbb R}_p^n$ be 
smooth and strictly convex bodies. Then
\begin{equation*} \xi_T(K_1 + K_2) \geq   \xi_T(K_1) + \xi_T(K_2). \end{equation*}
To prove this, let 
$T_1, T_2 \subseteq {\mathbb R}_p^n$ be smooth convex bodies.
From Theorem~\ref{BM-non-smooth-case}, and the fact that $$ (K_1 + K_2) \times (T_1 + T_2) = (K_1 \times T_1) + (K_2 \times T_2) $$
 it follows that:
\begin{equation*} \label{eq-bm-for-minkowski-billiard}
\widetilde c_{_{EHZ}}  \bigl ((K_1 + K_2)\times (T_1 + T_2) \bigr )^{\frac 1 2}  \geq
\widetilde c_{_{EHZ}}   (K_1 \times T_1)^{\frac 1 2}  + \widetilde c_{_{EHZ}}  ( K_2 \times T_2)^{\frac 1 2}.
\end{equation*}
Moreover, equality holds if and only if  $K_1 \times T_1$ and $K_2 \times T_2$ have a pair of capacity
 carriers which coincide up to translation and dilation. 
Let $T_1 = T_2 = T$, 
then, for any $\lambda \in (0,1)$ one has 
$$  \widetilde c_{_{EHZ}}  \bigl ((\lambda K_1 + (1-\lambda)  K_2)\times T \bigr )^{\frac 1 2}  \geq
\widetilde c_{_{EHZ}}   (\lambda K_1 \times \lambda T)^{\frac 1 2}  + \widetilde c_{_{EHZ}}  ( (1-\lambda) K_2 \times (1-\lambda) T)^{\frac 1 2}.$$
Using the homogeneity of $\widetilde c_{_{EHZ}}$, the fact that $\xi_T(K) = \widetilde c_{_{EHZ}}(K \times T)$,
and the weighted arithmetic-geometric mean inequality,  we obtain that for  $\lambda \in (0,1)$ 
$$    \xi_T (\lambda K_1 + (1-\lambda)  K_2) \geq   \xi_T(K_1)^{\lambda} \xi_T(K_2)^{1-\lambda}$$
In particular, this implies that for $K_1'= \lambda^{-1} K_1$, and $K_2'= (1 -\lambda)^{-1}K_2$ one has
$$ \xi_T(K_1 + K_2) = \xi_T(\lambda K_1'\ + (1-\lambda)K_2') \geq  \xi_T(K_1')^{\lambda} \xi_T(K_2')^{1-\lambda}$$
Next, we choose $\lambda \in (0,1)$ such that $\xi_T(K_1') = \xi_T(K_2')$. For this choice of $ \lambda$ we have
\begin{equation*} \xi_T(K_1 + K_2) \geq  \xi_T(K_1')^{\lambda} \xi_T(K_2')^{1- \lambda} =
 {\lambda} \xi_T (K_1') + (1- \lambda) \xi_T(K_2') = \xi_T(K_1) + \xi_T(K_2). \end{equation*}
To conclude, we obtain the following Brunn-Minkowski type inequality for $(K,T)$-billiard trajectories:
\begin{equation} \label{ex-desired-bm-ineq} \xi_T(K_1 + K_2) \geq  \xi_T(K_1) + \xi_T(K_2). \end{equation}
In particular, in the Euclidean case where $T=B$ we obtain the inequality in Theorem \ref{BM-for-Bill} for smooth and strictly convex $K_i$. 
The not-strictly convex case follows immediately from 
the fact that $\xi_T(K)$ is continuous with respect to the Hausdorff topology on
the class of convex bodies.

For the equality case in Theorem \ref{BM-for-Bill}, note that equality holds if and only if  $K_1 \times B$ and $K_2 \times B$ have a pair of capacity
 carriers $\gamma_1, \gamma_2$ which coincide up to translation and dilation. In particular, $\pi_p(\gamma_i)$  are   $K_i$-billiard trajectories in $B$ ($i=1,2$) which coincide up to translation and dilation. 
 It is not hard to check that this is possible if and only if
they coincide, which by duality means that 
their exists a closed curve which, up to translation, is a
length-minimizing billiard trajectory in both $K_1$ and $K_2$. 
This together with~$(\ref{ex-desired-bm-ineq})$ completes the proof of Theorem~\ref{BM-for-Bill}.

\subsection{Bounding $\xi(K)$ in terms of ${\rm Vol}(K)$.}
The main ingredient we need for the proof of this theorem is the following result from~\cite{AMO}, which provides a
dimension-independent bound for a symplectic capacity of a convex body in terms of its volume radius.

\begin{theorem} \label{thm-AAOM}
There exists a positive universal constant $A_0$ such that for every even dimension $2n$,
any symplectic capacity $c$, and any convex body $\Sigma \subset {\mathbb R}^{2n}$,
 one has $$ {\frac {c(\Sigma)} {c(B^{2n})} } \leq A_0 \Bigl (  {\frac {  {\rm Vol}(\Sigma)}
{   {\rm Vol}(B^{2n}) } }  \Bigr )^{\frac 1 n},$$
where $B^{2n}$ is the Euclidean unit ball in ${\mathbb R}^{2n}$.
\end{theorem}
From Theorem~\ref{thm-AAOM}  it follows that for any two convex bodies $K \subset {
\mathbb R}^n_q$, and $T \subset {\mathbb R}^n_p$, 
$$ \widetilde c_{_{EHZ}}(K \times T) \leq  A_1 n
 {\rm Vol}(K)^{\frac 1 n} {\rm Vol}(T)^{\frac 1 n},$$
for some positive universal constant $A_1$.
In particular, for $T=B^n \subset {\mathbb R}^n_p$ we obtain:
$$ \xi(K) =\widetilde  c_{_{EHZ}}(K \times B) \leq A_2 \sqrt{n} \, {\rm Vol}(K)^{\frac 1 n},$$
for some positive universal constant $A_2$. The proof of
Theorem~\ref{main-application2} is thus complete.
\subsection{Bounding $\xi(K)$ in terms of ${\rm inrad}(K)$.}
Let $K^* = {\frac 1 2} (K-K)$ denote the Minkowski symmetral of $K$ i.e.,
half of  the difference body. From Theorem~\ref{BM-for-Bill} and
 Theorem~\ref{thm-ghomi}  it follows  that:
\begin{equation} \label{first-estimate-on-xi-K}
\xi(K) = {\frac {\xi(K) + \xi(-K)} {2} } \leq \xi (K^*) = 4 \, inrad (K^*).
\end{equation}
Next, in order to compare $inrad(K^*)$ with $inrad(K)$, we shall use a
classical result from convexity  (see~\cite{Sch1,Ste})
which states that for a convex body $K$, one has
\begin{equation*}  
 width(K) \leq (n+1)  inrad(K),
\end{equation*}
 where $width(K)$ is the minimal width of the body $K$.
Combining 
this inequality with the (easily verified) fact that $width(K^*) = width(K)$, we conclude that
\begin{equation} \label{comparison-of-inradii}
inrad(K^*) = {\frac 1 2} width(K^*) = {\frac 1 2} width(K) \leq {\frac { (n+1) } 2 } inrad(K).
\end{equation}
Theorem~\ref{main-application} now follows from~$(\ref{first-estimate-on-xi-K})$
and~$(\ref{comparison-of-inradii}$).
 
\subsection{Monotonicity property of  billiards}
As before, we shall prove the monotonicity property in the more general case of Minkowski billiards.
It follows immediately from the fact that $\widetilde c_{_{EHZ}}$ is a symplectic capacity
that for any smooth convex bodies $K_1,K_2 \subset {\mathbb R}^n_q$, and $T_1,T_2
\subseteq {\mathbb R}^n_p$, such that $K_1 \times T_1 \subseteq K_2 \times T_2$, one has
$c_{_{EHZ}}(K_1 \times T_1) \leq  c_{_{EHZ}}(K_2 \times T_2)$.  This implies  
that
the length of the shortest periodic $T_1$-billiard trajectory in $K_1$ is less than or equal to the length of the shortest periodic $T_2$-billiard trajectory in $K_2$;
where the lengths of these trajectories is measured with respect to the gauge functions $h_{T_1}$, $h_{T_2}$ respectively. 
In particular, when $T_1=T_2 = T$, one has $\xi_T(K_1) \leq \xi_T(K_2)$. The case $T=B$, where $B$ is the Euclidean unit ball
in ${\mathbb R}^n_p$, completes the proof of
Proposition~\ref{prop-about-monotonicity}.

\section{A gliding-proper dichotomy}\label{sec:gpdich}

Here we prove Propositions~\ref{dichotomy-prop}.

\begin{proof}[\bf Proof of Proposition~\ref{dichotomy-prop}]
Let $\gamma(t)=(q(t),p(t)), t \in S^1$ be a $(K,T)$-billiard trajectory, and assume
 that $\gamma(t_0) \in \partial K \times \partial T$ for some $t_0 \in S^1$.

 \noindent {\bf Step I:} Assume that $\gamma(t_0) \in  (\partial K \times \partial  T)
 \setminus {\cal A}$. Then, we claim that there exists $\varepsilon_0 > 0 $ such that
  for every $0 < \varepsilon < \varepsilon_0$ one has $\gamma(t_0 \pm \varepsilon)
  \notin \partial K \times \partial T$. Indeed, it follows from~$(\ref{eq-the-cone})$ that
$\gamma$ has left and right derivatives at $t_0$ and
 $\dot \gamma^{\pm}(t_0)$  belong to the positive  cone spanned by $(0,\nabla g_K(\gamma_q(t_0)))$ and
    $(- \nabla g_T(\gamma_p(t_0)),0)$ i.e., there exists constants $\alpha_{\pm}, \beta_{\pm}$ such that
 \begin{equation*}
\dot \gamma_{\pm} (t_0) =   \alpha_{\pm} (-\nabla g_T(\gamma_p(t_0)) ,0) + \beta_{\pm}
(0,\nabla g_K(\gamma_q(t_0))),
\end{equation*}
where $\alpha_{\pm},\beta_{\pm} \geq 0$ and  $(\alpha_{\pm}, \beta_{\pm}) \neq (0,0)$.
In particular, we conclude that for $\varepsilon > 0$:
\begin{equation*}
\bigl (\gamma_q(t_0 + \varepsilon), \gamma_p(t_0 + \varepsilon) \bigr ) =
 \bigl (\gamma_q(t_0), \gamma_p(t_0) \bigr) + \varepsilon
 \bigl (- \alpha_{+} \nabla g_T(\gamma_p(t_0)),  \beta_{+}
\nabla g_K(\gamma_q(t_0)) \bigr ) + o(\varepsilon).
\end{equation*}
Recall that for a  smooth convex body $\Sigma \subset {\mathbb R}^n$, and $u \in \partial \Sigma$  one has
 $\langle  v-u , \nabla g_{\Sigma}(u) \rangle \leq 0$ for every $v \in \Sigma$. Thus,
\begin{equation*}
\begin{aligned}
\left        \langle \bigl (\gamma_q(t_0) - \varepsilon \alpha_{+} \nabla g_T(\gamma_p(t_0))
 + o(\varepsilon) \bigr ) -  \gamma_q(t_0)    , \nabla g_K(\gamma_q(t_0))  \right \rangle &\leq 0 ,\\
       \left        \langle     \bigl ( \gamma_p(t_0) + \varepsilon \beta_{+} \nabla g_K(\gamma_q(t_0))
        + o(\varepsilon) \bigr )   - \gamma_p(t_0)  , \nabla g_T(\gamma_p(t_0))  \right \rangle  & \leq 0,
       \end{aligned}
\end{equation*}
 and hence
\begin{equation*} \label{eq-scalar-product-of-normals1}
\begin{aligned}
 \varepsilon \alpha_{+} \langle   \nabla g_T(\gamma_p(t_0))  ,  \nabla g_K(\gamma_q(t_0))     \rangle
   &\geq o(\varepsilon) ,\\
        \varepsilon  \beta_{+} \langle   \nabla g_T(\gamma_p(t_0))  ,  \nabla g_K(\gamma_q(t_0))     \rangle     & \leq o(\varepsilon).
       \end{aligned}
\end{equation*}
Dividing by $\varepsilon$ and taking the limit as $\varepsilon \rightarrow 0^{+}$ we obtain
\begin{equation*} \label{eq-scalar-product-of-normals}
\begin{aligned}
 \alpha_{+} \langle   \nabla g_T(\gamma_p(t_0))  ,  \nabla g_K(\gamma_q(t_0))     \rangle
   &\geq 0 ,\\
          \beta_{+} \langle   \nabla g_T(\gamma_p(t_0))  ,  \nabla g_K(\gamma_q(t_0))     \rangle     & \leq 0.
       \end{aligned}
\end{equation*}
Moreover, a straightforward computation shows that
\begin{equation}   \label{eq-about-var-of-gk}
 \begin{aligned}
g_K(\gamma_q(t_0 + \varepsilon))   & =  g_K(\gamma_q(t_0)) -
\varepsilon \alpha_{+}  \langle   \nabla g_K(\gamma_q(t_0))  ,
 \nabla g_T(\gamma_p(t_0))
   \rangle + o(\varepsilon)  \\ & =  1 - \varepsilon \alpha_{+}
    \langle   \nabla g_K(\gamma_q(t_0))  ,  \nabla g_T(\gamma_p(t_0))     \rangle + o(\varepsilon),
       \end{aligned}
\end{equation}
and,
\begin{equation} \label{eq-about-var-of-gt}
 \begin{aligned}
g_T(\gamma_p(t_0 + \varepsilon))  
=  1 + \varepsilon \beta_{+}  \langle   \nabla g_K(\gamma_q(t_0))  ,  \nabla g_T(\gamma_p(t_0))     \rangle + o(\varepsilon).
       \end{aligned}
\end{equation}
Next, assume without loss of generality that $ \langle   \nabla g_T(\gamma_p(t_0))  ,
 \nabla g_K(\gamma_q(t_0))     \rangle   > 0$
(the argument applies verbatim in the opposite case).  Under this assumption,
we see that  $\beta_{+}=0$. 
Therefore, by assumption, $\alpha_{+} > 0$,
and equation~$(\ref{eq-about-var-of-gk})$ implies that for small enough $\varepsilon$, $g_K(\gamma_q(t_0 + \varepsilon))  < 1$.
More precisely there is $\varepsilon_0 > 0$ such that
 $\gamma(t_0 + \varepsilon) \in {\rm int}(K) \times \partial T$, for all
 $0 < \varepsilon < \varepsilon_0$.  A similar argument shows that
 $\gamma(t_0 - \varepsilon) \in \partial K \times {\rm int}(T)$, for $0 < \varepsilon < \varepsilon_0$.
Thus, we conclude that any $(K,T)$-billiard trajectory that enters the boundary
$\partial K \times \partial T$ at a point not in ${\cal A}$
must instantly exit the boundary $\partial K \times \partial T$.

 \noindent {\bf Step II:} Next we show that a $(K,T)$-billiard trajectory $\gamma$ which enters the
 boundary $\partial K \times \partial T$ must do so at a point not in ${\cal A}$.
Indeed, assume without loss of generality that $\gamma(t_0) = (\gamma_q(t_0),\gamma_p(t_0))
\in \partial K \times {\rm int}(T)$.
 Then, by definition, there exists $\varepsilon>0$ such that for every
 $t \in [t_0,t_0 + \varepsilon)$ one has $\dot \gamma(t) = (0, d \nabla g_K(\gamma_q(t_0)))$,
 for some positive constant $d$, and that at time $t_0 + \varepsilon$ the trajectory
 hits  the boundary $\partial K \times \partial T$.
Since $T$ is a smooth convex body, and
$\gamma_p(t_0 + \varepsilon)
\in \partial T$  one has (as $\gamma_p(t_0) \in {\rm int}(T)$), that
$$  \langle  \gamma_p(t_0 + \varepsilon)  - \gamma_p(t_0) ,
 \nabla g_{T}(\gamma_p(t_0 + \varepsilon)  \rangle > 0.$$ Thus, it follows that
  $  \varepsilon d \langle  \nabla g_K(\gamma_q (t_0 + \varepsilon)) ,
  \nabla g_{T}(\gamma_p(t_0 + \varepsilon)  \rangle > 0$. This completes the proof of Step II.


 \noindent {\bf Step III:}
Toward a contradiction, assume that $\gamma$ is a $(K,T)$-billiard trajectory
which intersects ${\cal A}$ but does not lie exclusively in it.
We claim that in such a case the trajectory must intersect $\partial K \times \partial
 T \setminus {\cal A}$ at a sequence of points which converges to a point in ${\cal A}$.
Indeed, consider the point $\gamma(t_{\tau})$ which is the last (first) point lying in ${\cal A}$.
That is, $t_\tau$ is maximal (minimal) with this property, where we use that the condition of belonging to ${\cal A}$ is closed.  
some point $\gamma(t)$ where  $ t_{\tau}<t < t_{\tau} + \varepsilon$ (thus lying outside
of ${\cal A}$). Either $\gamma(t)  \in \partial K \times \partial T \setminus {\cal A}$
of $\gamma(t) \not\in \partial K \times \partial T$. In the latter case, assume without loss of generality that $\gamma(t) \in int(K) \times \partial T$.
By Step II, there exists a point $t_{\tau} < t' < t$ such that $\gamma(t') \in \partial K
\times \partial T \setminus {\cal A}$.
In both cases we have found a time $t_{\tau} < t' < t_{\tau} + \varepsilon$ such that $\gamma(t') \in
\partial K \times \partial T \setminus {\cal A}$. Since this applies for every $\varepsilon$ we have
found a sequence  $t_i \rightarrow t_{\tau}$ such that $\gamma(t_i) \in \partial K \times
\partial T \setminus {\cal A}$. We next show that in such
a case one has $ \sum_{i=1}^{\infty} | q_{i+1} - q_i | < \infty$, and $ \sum_{i=1}^{\infty}
 | p_{i+1} - p_i | < \infty$, where $\gamma(t_i) = (q_i,p_i)$. Indeed, as $p_k \rightarrow p$ we may choose a coordinate system such that every coordinate of $\nabla g_T(p)$ is non zero.
We denote the $j^{th}$-coordinate of a vector $v$ by $v^j$. Therefore, we may also assume that $(\nabla g_T(p_k))^j \neq 0$ for all $k$ and $j$. Next we estimate
\begin{equation}\label{eq:telescop}
\sum_{i=1}^{\infty} | q_{i+1} - q_i | \leq  \sqrt{n} \sum_{i=1}^{\infty} \sum_{j=1}^n | q^j_{i+1} - q^j_i |. 
\end{equation}   
Since $q^j_{i+1} - q^j_i = Const (\nabla g_T(p_i))^j$ it is of constant sign for all $i$. The right hand side of (\ref{eq:telescop}) is thus  a telescopic sum and hence the left hand side of (\ref{eq:telescop})  converges.

To arrive at a contradiction, we use a  result by P. Gruber (Theorem 2 in \cite{Gruber}, for the planar case see~\cite{Halpern}) which states that whenever a convex body is $C^3$ smooth and has  positive Gaussian curvature, there exists no billiard trajectory $(q_i, p_i)$ for which 
the above two series converge. 
The theorem is proved for the Euclidean case but the proof can be naturally adjusted  to the Minkowski setting. 
Note that in the terminology of \cite{Gruber} this means that no trajectory ``terminates on the boundary''.

 \noindent {\bf Step IV:} Here we prove the second part of Proposition~\ref{dichotomy-prop}.
  From Step I we conclude that if $\gamma$ is a gliding trajectory, then
  $\gamma \in {\cal A}$ i.e., for every $t \in S^1$ one has
  \begin{equation} \label{eq-about-gamma-in-the-boundary}  \dot \gamma(t) =
   (\dot \gamma_q(t), \dot \gamma_p(t))=
 (- \alpha(t) \, \nabla g_{T}(\gamma_p(t)), \beta(t)  \nabla g_{K}(\gamma_q(t))), \end{equation}
where $\alpha(t)$ and $ \beta(t)$ are two smooth   positive real functions, and
\begin{equation}  \label{eq-perp-normal-vec1} \langle \nabla g_T(\gamma_p(t)),
 \nabla g_K(\gamma_q(t)) \rangle = 0. \end{equation}
Taking the time-derivative of~$(\ref{eq-perp-normal-vec1})$  we obtain
 \begin{equation}  \label{yet-another-eq-about-perp} \langle \nabla^2 g_T(\gamma_p(t))
  \dot \gamma_p(t), \nabla g_K(\gamma_q(t) )\rangle  +
 \langle \nabla g_T(\gamma_p(t)), \nabla^2 g_K(\gamma_q(t))  \dot \gamma_q(t) \rangle = 0. \end{equation}
Equality~$(\ref{condition-on-alpha-and-beta})$ now follows immediately from the combination
 of~$(\ref{eq-about-gamma-in-the-boundary})$ and~$(\ref{yet-another-eq-about-perp})$.
In order to prove that both $\alpha$ and $\beta$ are positive, we note that if e.g., $\alpha(t)=0$, then by~$(\ref{yet-another-eq-about-perp})$ one get 
that 
 \begin{equation}  \beta(t)  \langle \nabla^2 g_T(\gamma_p(t))
  \nabla g_K(\gamma_q(t)) , \nabla g_K(\gamma_q(t) )\rangle = 0. \end{equation}
However, the only vector in the kernel of $\nabla^2 g_T(\gamma_p(t))$ is $\nabla g_T(\gamma_p(t)) $ which is orthogonal to $ \nabla g_K(\gamma_q(t)) $.
Therefore $\beta(t)=0$. This is in contradiction to the assumption $(\alpha,\beta) \neq (0,0)$.
 This completes the proof of  Proposition~\ref{dichotomy-prop}.
\end{proof}

\section{Appendix}

In this section we provide the proofs of Proposition \ref{proposition-EHZ-capacity-non-smooth} and of Proposition \ref{prop-about-the-1-1-corrospendence}. The methods are quote similar to those used in the smooth case, and are repeated with the necessary changes for completeness. 

We shall use a bijection between 
generalized closed characteristics $\gamma$ on $\partial \Sigma$, and 
so called weak critical points of 
the functional 
$$I_{\Sigma}(z):=  \int_0^{2 \pi} h_{\Sigma}^2(\dot z(t)) dt,$$
defined on the loop space ${\cal E}$. The set $\Sigma \in {\widetilde {\cal K}}^{2n}$ will be chosen to be of the form $\Sigma= K \times T$.
This bijection is very similar to the one defined in~\cite{AAO1}, however, an adaptation is needed since in~\cite{AAO1} only smooth hypersurfaces were considered (and there was no need to consider weak solutions of the equation for critical points). 
The space of weak critical points of the functional $I_{\Sigma}$ is defined by 
$$ {\cal E}^{\dagger} := \{ z \in {\cal E} \ ; \  {\rm there \ is \ } \alpha \in {\mathbb R}^{2n} \ {\rm such \ that \ } I_{\Sigma}(z) J z + \alpha \in \partial h^2_{\Sigma}(\dot z) \}.$$
This can be compared with~\cite{AAO1}, where the inclusion sign is substituted by equality. 
Not surprisingly, it turns out that the minimum of $I_{\Sigma}(z)$ over the space ${\cal E}$ is attained on a weak critical point. This is proved in Lemma~\ref{Lemma-about-weak-critical} 
below.  The following lemma, explains the above mentioned bijection and its properties, is also proven below. 

\begin{lemma} \label{lemma-about-correspond-crit-pt}
Let  $\Sigma \in {\widetilde {\cal K}}^{2n}$. There is a bijection ${\cal F}$ between the set of
generalized closed characteristics $\gamma$ on $\partial \Sigma$, and the set of elements $z \in {\cal E}^{\dagger}$.
 Moreover, under this bijection, 
$A(\gamma) =  {\frac {\pi} 2} I_{\Sigma}( z)$, where
 $z = {\cal F}(\gamma)$. In particular, the minimum  of $I_{\Sigma}(z)$ via this bijection corresponds to the action minimizing generalized closed characteristic.
 \end{lemma}

\begin{proof}[{\bf Proof of Proposition~\ref{proposition-EHZ-capacity-non-smooth}}]
It is not hard to check that from the definition of  $\widetilde c_{_{EHZ}}$ (see Lemma~\ref{tech-lemma-about-extension}), 
and the fact that $ {\mathcal K}^{2n} $ is dense in $ \widetilde  {\mathcal K}^{2n}  $,  it follows that
\begin{equation}  \label{eq-ext-of-c-with-inf}
\widetilde  c_{_{\it EHZ}}(\Sigma)=  \inf_{z \in
{\cal E}}
  {\frac {\pi} {2 }} \int_0^{2 \pi} h_{ \Sigma}^{2}(\dot{z}(t))dt. \end{equation}

Thus, to conclude the proof of the  proposition it is enough to show that the infimum on the right hand side of~$(\ref{eq-ext-of-c-with-inf})$ is attained. 
To show this, we first claim that  
the functional $I_{\Sigma}$  is bounded from below
on ${\cal E}$, namely, there is a positive constant $\alpha$ such that $I_{\Sigma}(z) > \alpha$ for every $z \in {\cal E}$.
The proof of this fact can be found in~\cite{HZ}, Section 1.5 (cf.~\cite{AAO1}, Step I of Lemma 2.5).
Next we show that  
there is $\tilde z \in {\cal E}$ with
$$ \int_0^{2\pi} h_{\Sigma}^2 ( \dot{\tilde z}(t)) dt = \inf_{z \in
{\cal E}}  \int_0^{2\pi} h_{\Sigma}^2 ( \dot{z}(t)) dt =: \tilde \lambda >
0.$$ 
We follow closely the arguments of Step II in Lemma 2.5 from~\cite{AAO1} .
Pick a minimizing sequence $z_j \in {\cal E}$
such that
\begin{equation*} \lim_{j \rightarrow \infty} \int_0^{2\pi} h_{\Sigma}^2 ( \dot{z_j}(t))dt = \tilde \lambda. \end{equation*}
It follows from Step II of Lemma 2.5 in~\cite{AAO1} that there is a constant $C>0$ such that
\begin{equation} \label{some-bound-on-subsequence}
\| z_j \|_2 \leq 2 \pi \| \dot z_j \|_2 \leq C
\end{equation}
Moreover, it was shown ibid that there is a subsequence, also denoted by $ z_j $, which converges  both uniformly and weakly  in $W^{1,2}(S^1,{\mathbb R}^{2n})$
to an element 
$z_* \in {\cal E}$.
Hence, it is enough to show that $z_*
\in {\cal E}$ is indeed the required minimum. From the convexity
of $h_{\Sigma}^2$ we  deduce the point-wise estimate
$$  h_{\Sigma}^2({\dot z}_*(t)) -
h_{\Sigma}^2({\dot z}_j(t)) \leq \langle \partial h_{\Sigma}^2({\dot z}_*(t)), {\dot
z}_*(t)-{\dot z}_j(t) \rangle,$$
where the inequality hold in the set-theoretical sense, i.e, for any
element in the corresponding subdifferential. 

We need to make a measurable choice of elements in $\partial h_{\Sigma}^2({\dot z}_*(t))$, and show that the chosen function belongs to $L_2(S^1,{\mathbb
R}^{2n})$.  It is not hard to check that 
there exists some positive constant $C$ for which $| \partial
h_{\Sigma}^2(x)| \leq C |x|$ for all $x$, so that (since ${\dot z}_*(t)$ belongs to $L_2(S^1,{\mathbb
R}^{2n})$) boundedness is not a problem. 
The fact that there exist a measurable choice from the subgradient follows from a standard measure theoretic argument. 

For this choice we have
\begin{equation} \label{eq1000} \int_0^{ 2 \pi} h_{\Sigma}^2 ({\dot z}_*(t))dt - \int_0^{ 2
\pi} h_{\Sigma}^2 ({\dot z}_j(t))dt \leq \int_0^{ 2 \pi} \langle \partial
h_{\Sigma}^2({\dot z}_*(t)), {\dot z}_*(t)-{\dot z}_j(t) \rangle dt
\end{equation}
and the
right hand side of inequality~$(\ref{eq1000})$ tends to zero. Hence,
$$ \tilde \lambda \leq \int_{0}^{2 \pi} h_{\Sigma}^2( {\dot z}_*(t)) dt
\leq \liminf_{j \rightarrow \infty} \int_{0}^{2 \pi} h_{\Sigma}^2( {\dot
z}_j(t)) dt = \tilde \lambda,$$ and we have proved that $z_*$ is the
minimum of $I_{\Sigma}(z)$ on $ {\cal E}$.
\end{proof}

As was mentioned above, the minimum of $I_{\Sigma}$ ensured by the above proposition is in fact attained on a weak critical point, namely, 
a point $z \in {\cal E}^{\dagger}$, as we next prove. 
\begin{lemma} \label{Lemma-about-weak-critical}  Let $\Sigma \in {\widetilde {\cal K}}^{2n}$.
For any minimizer $z \in {\cal E}$ of 
the functional $I_{\Sigma}$  there is $\alpha \in {\mathbb R}^n$ such that $I_{\Sigma}(z)Jz + \alpha \in \partial h^2_{\Sigma}(\dot z)$.
\end{lemma}

\begin{proof}
We will show that any minimizer $z \in {\cal E}$ of 
the functional $I_{\Sigma}$  (the existence of which is ensured  by Proposition~\ref{proposition-EHZ-capacity-non-smooth}) satisfies $z \in {\cal E}^{\dagger}$.
To prove this fact, let $z \in {\cal E}$ be a minimizer of $I_{\Sigma}$. Note that addition of a constant (vector) to $z$ does not change the action nor the functional $I_{\Sigma}$. 
Thus, we may ignore the normalization condition $\int_{0}^{2\pi} z(t)dt =0$. We next use some standard tools from non-smooth analysis namely non-smooth versions of Euler-Lagrange equations and of Lagrange multipliers theory. It follows from~\cite{Cl2} that there are constants $\alpha_1, \alpha_2$, not both equal to zero such that 
$$ 0 \in  \partial \left [ \int_0^{2 \pi}  \alpha_1 h^2_{\Sigma} (\dot z(t))  + \alpha_2   \langle Jz(t), \dot z(t) \rangle    dt  \right ]. $$
Here, the subdifferential $\partial F$, where $F : W^{1,2} (S^1,{\mathbb R}^{2n}) \rightarrow {\mathbb R}$ is in the sense of Clarke (see e.g.~\cite{Cl1}).
In particular, all directional derivatives (again, in the sense of Clarke) are non-negative. In what follows, we follow closely the reasoning in the proof of Theorem 2.2 in~\cite{Cl1}. 
We denote by $G(x,y) =  \alpha_1 h^2_{\Sigma} (y)  + \alpha_2   \langle Jx, y \rangle    $.
The information about directional derivatives can be written as:
\begin{eqnarray*}
0 &  \leq & \liminf_{\lambda \downarrow 0} \int_0^{2 \pi}   {\frac { G \bigl( z(t) + \lambda h(t) ,    \dot z(t) + \lambda  \dot h(t)   \bigr ) - G(z(t),\dot z(t) )} {\lambda}}    dt \\
  &\leq &  \int_0^{2 \pi}  \limsup_{\lambda \downarrow 0}   {\frac { G \bigl( z(t) + \lambda h(t) ,    \dot z(t) + \lambda  \dot h(t)   \bigr ) - G(z(t),\dot z(t) )} {\lambda}}    dt  \\
 & \leq & \int_0^{2 \pi} \max_{(v_1,v_2)  \in  \partial G(z(t),\dot z(t))}   \left (  h(t) v_1 + \dot h(t) v_2  \right ) dt   \\
& = & \max_{A}  \int_0^{2 \pi}  \left (  h(t) v_1(t) + \dot h(t) v_2(t)  \right ) dt,
\end{eqnarray*}
where the last maximum is taken over the set $A$ of  all measurable selections $(v_1(t),v_2(t)) \in \partial G(z(t),\dot z(t))$.
From this one concludes (as $h=0$ is a valid test function) that 
$$ \min_{h \in W^{1,2}(S^1,{\mathbb R}^{2n} )}  \max_{A}  \int_0^{2 \pi}  \left (  h(t) v_1(t) + \dot h(t) v_2(t)  \right ) dt = 0.$$
As in~\cite{Cl1}, we invoke a min-max theorem from~\cite{Cl3} to switch the minimum and the maximum, thus conclude that 
there is some element $(v_1(t),v_2(t)) \in A$ for which 
$$ \int_0^{2 \pi}  \left (  h(t) v_1(t) + \dot h(t) v_2(t)  \right ) dt \geq  0,$$
for all $h \in W^{1,2}(S^1,{\mathbb R}^{2n} )$. From linearity in $(h,\dot h)$ the above expression must vanish for all $h \in W^{1,2}(S^1,{\mathbb R}^{2n} )$.
We  next compute $\partial G$:
$$ \partial G(x,y) =  \left (  - \alpha_2 Jy,  \alpha_1 \partial h^2_{\Sigma} (y)  + \alpha_2  Jx \right) .$$ 
Therefore $v_1(t) =  - \alpha_2 J \dot z(t)$ and $v_2(t)  \in  \alpha_1 \partial h^2_{\Sigma} (\dot z(t) )  + \alpha_2  Jz(t)$. 
The above argument implies that for all $h \in W^{1,2}(S^1,{\mathbb R}^{2n} )$ one has
\begin{equation} \label{last-eq-we-need} \int_0^{2 \pi}  \left (  - \alpha_2  h(t) \cdot  J \dot z(t) + \dot h(t) \cdot  (   \alpha_1 \eta(t) + \alpha_2 Jz(t) )   \right ) dt =0,  \end{equation}
for some $\eta(t) \in \partial h^2_{\Sigma}(\dot z(t))$. Using integration by parts, one rewrites this as
$$ \int_0^{2 \pi}   \dot h(t) \cdot  \bigl( \alpha_1 \eta(t) +  2 \alpha_2 Jz(t) \bigr )   dt =0.$$
Since this equality holds for every $h \in W^{1,2}(S^1,{\mathbb R}^{2n} )$, one concludes that there is some constant vector $\alpha$ 
for which 
$$   \alpha_1 \eta(t) +  2 \alpha_2 Jz(t)  = \alpha.$$
Note that that $\alpha_1 \neq 0$  (as expression~$(\ref{last-eq-we-need}$) is linear in $\alpha_2$), and hence the above equation can be restated as: for some constants $\mu \in {\mathbb R}$ and $\alpha \in {\mathbb R}^{2n}$ one has
$$ \alpha + \mu J z \in \partial h^2_{\Sigma}(\dot z).$$
Moreover, it follows from the
Euler formula (see~\cite{Fuc-Zh}) that
$$ I_{\Sigma}(z) = \int_0^{2 \pi} h_{\Sigma}^2({\dot z}(t))dt = {\frac 1 2} \int _0^{2 \pi}
\langle  \partial  h_{\Sigma}^2({\dot z}(t)) , {\dot z}(t) \rangle dt =
{\frac \mu 2} \int_0^{2\pi} \langle Jz(t), {\dot z}(t) \rangle
dt = {\mu}.$$ 
This shows that $z \in {\cal E}^{\dagger}$, and the proof of  Lemma~\ref{Lemma-about-weak-critical} is now complete. 
\end{proof}


We next prove Lemma \ref{lemma-about-correspond-crit-pt}, which was mentioned at the beginning of the appendix.

\begin{proof}[{\bf Proof of Lemma~\ref{lemma-about-correspond-crit-pt}}]
To define the mapping ${\cal F}$, let $z \in {\cal E}$  such that  \begin{equation} \label{Euler_eqt}
\partial h_{\Sigma}^2 ({\dot {z}}) \ni   \lambda J \, { z} +
\alpha,
\end{equation}
for some vector $\alpha$ and $\lambda = I_{\Sigma}(z)$. We will use the
Legendre transform in order to define an affine linear image of
$z$ which is a generalized closed characteristic on the boundary $\partial \Sigma$,
 which we will then define as ${\cal F}^{-1}(z)$. Recall that
the Legendre transform is defined as follows: For $f : {\mathbb
R}^{n} \rightarrow {\mathbb R}$, one defines
$$ {\cal L} f(y) = \sup_{x \in {\mathbb R}^n} [ \langle y , x
\rangle - f(x) ], \ \forall y \in {\mathbb R}^n.$$
It is not hard to
check that
$ ({\cal L} (h_{\Sigma}^2))(v) =
4^{-1} h_{\Sigma^\circ}^2 (v), $ where $\Sigma^{\circ}$
is the polar body of $\Sigma$. Note that $h^2_{\Sigma^{\circ}}$ is
a defining function of $\Sigma$ (that is, $\Sigma$ is its 1-level set) which
is homogeneous of degree $2$.
 After applying the Legendre transform
and using the fact that $v =  \partial h_{\Sigma}^2(u)$ is inverted
point-wise (as a multivalued function) by $u = \partial {\cal L} h_{\Sigma}^2(v)$ (see e.g.,~\cite{Roc} Corollary 23.5.1)  equation
$(\ref{Euler_eqt})$ becomes:
\[ {\dot {z}} \in  4^{-1} \partial h_{\Sigma^{\circ}}^2 
(  \lambda J \,  z + \alpha )  = \partial h_{\Sigma^{\circ}}^2 \bigl
(  4^{-1} (\lambda J \,  z + \alpha) \bigr).
\]
 Next, let
\begin{equation} \label{eq_20} \gamma = {\kappa} \bigl
(  4^{-1} (\lambda J \,  z + \alpha) \bigr) 
\end{equation}
where $\kappa$ is a positive normalization constant which we will
readily choose. 
Differentiating~$(\ref{eq_20})$ we see  that $\gamma$
satisfies the following: 
\begin{equation*}
{\dot {\gamma}} \in  {\kappa } \, 4^{-1} \lambda \, J \partial h_{{\Sigma}^{\circ}}^2 ( { {\gamma}/ \kappa})
=4^{-1} \, \lambda
J \partial h_{{\Sigma}^{\circ}}^2 (
 \gamma).\end{equation*}
Since we ask $\gamma \in {\partial \Sigma}$ we need to
choose $\kappa$ such that $\gamma$  lie in the energy level
$h_{{\Sigma}^{\circ}}^2=1$. 
Since $h_{\Sigma^{\circ}}^2$ is homogeneous of degree $2$ we obtain from
Euler's formula (see~\cite{Fuc-Zh}) that
\begin{eqnarray*} {\frac 1 {2 \pi}} \int_0^{2\pi} h_{{\Sigma}^{\circ}}^2(
 \gamma(t))dt  & = & {\frac 1 { 4 \pi } } \int_0^{2\pi} \langle \partial
h_{{\Sigma}^{\circ}}^2(  \gamma(t)) ,  \gamma(t) \rangle dt = -{\frac {
1 } { \pi  \lambda } } \int_0^{2\pi} \langle J
\dot{ \gamma}(t) ,\gamma(t) \rangle dt
\\ & = & {\frac { \kappa^{2}  } { 16 \pi   \lambda}}
\int_0^{2\pi} \langle   \lambda {\dot {{z}}}(t) , \lambda J z(t) +
\alpha \rangle dt = {\frac { \kappa^{2}
 \lambda } {8 \pi   }},
\end{eqnarray*}
which is equal to $1$ if we choose $\kappa = ({ {8 \pi } / {
\lambda}})^{\frac 1 2}  $. For this
value of $\kappa$ one has that
\begin{equation}  \label{new_eq_for_l-correspond} 
  \gamma
= \Big( {\frac {\pi} {2 \lambda}} \Big )^{\frac 1 2} \Big( {
\lambda } Jz + { {\alpha} } \Big), \end{equation} is a generalized closed characteristic on $\partial \Sigma$. 
This $\gamma$ we denote by
${\cal F}^{-1}(z)$. Before we turn to show that this mapping is
invertible, 
we first derive the relation between ${\cal A}(\gamma)$ and $\lambda =
I_{\Sigma}(z)$. Using Euler's formula, and the above value of $\kappa$ one has:
 $$ {\cal A} (\gamma) = {\frac 1 2} \int_0^{2 \pi} \langle -J {\dot \gamma}(t) ,\gamma(t) \rangle dt =
  {\frac {\kappa^2} {32}}  {\lambda}^2  \int_0^{2 \pi} \langle
  {\dot z}(t) ,  J  z(t) + {\frac { \alpha} {\lambda}} \rangle dt
 = {\frac {\pi \lambda} {2} }   = {\frac \pi 2} I_{\Sigma}(z). 
$$
In order to show that the map  ${\cal F}^{-1}$ is indeed one-to-one and onto, we now
define ${\cal F}$. We start with a generalized closed characteristic on $\partial \Sigma$ which is the
image of
 a loop $\gamma$ where $\gamma:[0,2 \pi] \rightarrow {\mathbb R}^{2n}$
and $\dot \gamma \in d J \partial h_{{\Sigma}^{\circ}}^2(\gamma)$, for some constant $d$.
Next, we define
\[{\cal F}(\gamma) = J^{-1} \Bigl ((2 \pi d )^{-1/2} \Bigl (\gamma - {\frac
1 {2\pi}} \int_0^{2\pi} \gamma(t)dt \Bigr ) \Bigr ),\] 
and set $z={\cal F}(\gamma)$. It is easy to check that $\int_0^{2 \pi}
z(t)dt=0$. The fact that $z \in W^{1,2}(S^1, {\mathbb R}^{2n})$
follows from the boundedness of $\gamma$ (as ${\rm Image}(\gamma) \in \partial \Sigma$ is bounded) 
and the following argument: since $\dot z \in {C_1} \partial h_{{\Sigma}^{\circ}}^2(\gamma)$ for some constant $C_1$, and 
$|\partial h_{{\Sigma}^{\circ}}^2(x)| \le C_2 |x|$  for some constant
$C_2$ (again, in the set-theoretical sense), we conclude that for some constants $C_3$ 
\[
\int_0^{2\pi} |\dot{z}(t)|^2dt  
\le C_3 \int_0^{2\pi}
|\gamma(t)|^2dt < \infty.
\]
Moreover, a direct computation shows that:
\begin{eqnarray*} {\cal A}(z) & = & {\frac 1 2} \int_0^{2\pi}
\langle \dot z(t), Jz(t) \rangle dt = {\frac 1 2} (\pi d 2)^{-1}
\int_0^{2\pi} \langle J^{-1} \dot \gamma(t), \gamma(t) \rangle dt \\ & = &
{\frac 1 2} (\pi d 2)^{-1} \int_0^{2\pi} \langle d \partial
h_{{\Sigma}^{\circ}}^2(\gamma(t)), \gamma(t) \rangle dt = {\frac 1 {2 \pi}}
\int_0^{2\pi} h_{{\Sigma}^{\circ}}^2(\gamma(t)) =1,
\end{eqnarray*}
where the next to last inequality follows from Euler's formula.
Finally, note that
\[ \dot z = (\pi d 2)^{-\frac {1} 2} J^{-1} \dot \gamma \in  (\pi
d 2)^{-\frac {1} 2} d \partial h_{{\Sigma}^{\circ}}^2(\gamma)= (\pi d 2)^{-\frac
{1} 2} d \partial h_{\Sigma^{\circ}}^2 \Bigl((\pi d 2)^{\frac 1
2}Jz+\frac{1}{2\pi} \int_0^{2\pi} \gamma(t)dt \Bigr).\]
Using the Legendre transform as before we get that
$$ \partial h_{\Sigma}^2( \dot z) \ni \lambda Jz + \alpha,$$ where
$\alpha \in {\mathbb R}^{2n} $ and $\lambda \in {\mathbb R}$ are constants. 
From the homogeneity of $h^2_{\Sigma}$ we conclude:
$$ I_{\Sigma}(z) = \int_0^{2 \pi} h_{\Sigma}^2(\dot{z}(t))dt = {\frac 1 2}
\int_0^{2 \pi} \langle \partial h_{\Sigma}^2(\dot{z}(t)), \dot z(t) \rangle
dt = {\frac {\lambda} 2} \int_0^{2 \pi} \langle  Jz(t), \dot z(t)
\rangle dt =  \lambda.$$
A straightforward computation shows (we omit the details) that for $\gamma$ and $z$ as above one has 
${\cal F}^{-1} {\cal F} (\gamma) = \gamma$ and ${\cal F} {\cal F} ^{-1}(z) = z$.  

To complete the proof of the lemma, we use Lemma~\ref{Lemma-about-weak-critical} which implies that the minimum of $I_{\Sigma}(z)$ is attained 
on a weak critical point and thus its corresponding generalized closed characteristic must be an action minimzer.
The proof of the lemma is now complete.
\end{proof}

\begin{proof}[{\bf Proof of Proposition~\ref{prop-about-the-1-1-corrospendence}}]
Note that the equality on the left hand side of~$(\ref{ANS1}$) follows from Proposition~\ref{proposition-EHZ-capacity-non-smooth}. Combining this fact with   
Lemma~\ref{lemma-about-correspond-crit-pt} we conclude  that:
$$ \widetilde  c_{_{EHZ}}(\Sigma) \leq   \min_{z \in {\cal E}^{\dagger}}  {\frac {\pi} 2} I_{\Sigma}(z)   =   \min {\widetilde {\cal L}}(\Sigma),$$
and since the minimum is attained on a weak critical point by Lemma~\ref{Lemma-about-weak-critical}, the proof of the proposition is now complete.

\end{proof}

\bigskip

\noindent Shiri Artstein-Avidan\\
\noindent School of Mathematical Science, Tel Aviv University, Tel Aviv, Israel\\
\noindent {\it e-mail}: shiri@post.tau.ac.il
\bigskip

\noindent Yaron Ostrover\\
\noindent School of Mathematical Science, Tel Aviv University, Tel Aviv, Israel\\
\noindent {\it e-mail}: ostrover@post.tau.ac.il


\begin{thebibliography}{}


\bibitem{AMO}Artstein-Avidan, S., Milman, V., Ostrover, Y. {\it The M-ellipsoid,
Symplectic Capacities and Volume}, Comm. Math. Helv., Volume 83,
issue 2, pages 359-369.

\bibitem{AAO1} Artstein-Avidan, S., Ostrover Y. {\it  Brunn-Minkowski inequality for
symplectic capacities of convex domains},   Int. Math. Res. Not.
Vol. 2008, no. rnn044, (2008).

\bibitem{Al} Albers, P. {\it Private communication.}

\bibitem{AlM} Albers, P., Mazzucchelli, M. {\it Periodic bounce orbits of
prescribed energy}, Int. Math. Res. Notices 2011 2011: 3289-3314.

\bibitem{BG} Benci, V., Giannoni,  F. {\it Periodic bounce trajectories with a low number of bounce points,}  Ann. Inst.
Henri Poinca\'re,  Anal.  Non Lin\'eaire, 6, No.1, 73-93, 1989.

\bibitem{BL} Brascamp, H.J., Lieb,  E.H. {\it On extensions of the
Brunn-Minkowski and Pr\'{e}kopa-Leindler theorem, including
inequalities for log-concave functions, and with an application to
the diffusion equation,}  J. Funct. Anal., 22  (1976), 366-389.

\bibitem{CHLS} Cieliebak, K., Hofer, H., Latschev, J., Schlenk
F. {\it Quantitative symplectic geometry,} Dynamics, ergodic theory,
and geometry, 1-44, Math. Sci. Res. Inst. Publ., 54, Cambridge Univ.
Press, Cambridge, 2007.

\bibitem {Cl} Clarke, F. H. {\it A classical variational principle for periodic Hamiltonian
trajectories,} Proc. Amer. Math. Soc. 76 (1979), no. 1, 186-188.

\bibitem{Cl1} Clarke, F.H. {\it Methods of dynamic and nonsmooth optimization,}
CBMS-NSF Regional Conference Series in Applied Mathematics, 57. Society for Industrial and Applied Mathematics (SIAM), 
Philadelphia, PA, 1989.

\bibitem{Cl2} Clarke, F.H. {\it  A new approach to Lagrange multipliers,}  Math. Oper. Res. 1 (1976), no. 2, 165-174.

\bibitem{Cl3} Clarke, F.H. {\it Multiple integrals of Lipschitz functions in the calculus of variations,} 
Proc. Amer. Math. Soc. 64 (1977), no. 2, 260-264. 

\bibitem{CdV} Colin de Verdi\`ere, Y. {\it
Spectrum of the Laplace operator and periodic geodesics: thirty
years after,}
Ann. Inst. Fourier (Grenoble) 57 (2007), no. 7, 2429-2463.


\bibitem{EH} Ekeland, I., Hofer, H. {\it Symplectic topology
and Hamiltonian dynamics}, Math. Z. 200 (1989), no. 3, 355--378.

\bibitem{EH1} Ekeland, I., Hofer, H. {\it Symplectic topology
and Hamiltonian dynamics II}, Math. Z.  203 (1990), no.4, 553--567.


\bibitem{FGS} Frauenfelder, U., Ginzburg, V., Schlenk, F. {\it Energy capacity
inequalities via an action selector},  Geometry, spectral theory,
groups, and dynamics, 129-152, Contemp. Math., 387, Amer. Math.
Soc., Providence, RI, 2005.

\bibitem{Fuc-Zh} Fuchun, Y., Zhou, W. {\it Generalized Euler identity for subdifferentials of homogeneous functions and applications}, 
J. Math. Anal. Appl. 337 (2008), no. 1, 516-523. 


\bibitem{Gh} Ghomi, M. {\it Shortest periodic billiard trajectories in convex bodies,}
Geom. Funct. Anal. 14 (2004), no. 2, 295-302.

\bibitem{G} Gromov, M. {\it Pseudoholomorphic curves in symplectic manifolds,}
Invent. Math. 82 (1985), no. 2, 307-347.

\bibitem{Gruber} Gruber, P.  M. {\it Convex billiards}, Geom. Dedicata 33 (1990), no. 2, 205-226.


\bibitem{GT} Gutkin, E., Tabachnikov, S. {\it Billiards in Finsler and Minkowski geometries,}
J. Geom. Phys. 40 (2002), no. 3-4, 277-301.

\bibitem{GM} Guillemin, V., Melrose, R. {\it The Poisson summation formula for manifolds with boundary,}
 Adv. Math. 32 (1979), 204-232.



\bibitem{Halpern} Halpern, B.  {\it Strange billiard tables}, Trans. Amer. Math. Soc. 232 (1977), 297--305. 


\bibitem{Her} Hermann, D. {\it Private communication.}

\bibitem{H1} Hofer, H. {\it On the topological properties of symplectic
maps,} Proc. Roy. Soc. Edinburgh Sect. A 115, 25-38 (1990).

\bibitem{HZ} Hofer, H., Zehnder, E. {\it Symplectic Invariants
and Hamiltonian Dynamics}, Birkh\"auser, Basel (1994).


\bibitem{HZ1} Hofer, H., Zehnder, E.  {\it A new capacity for symplectic
manifolds,} Analysis et cetera. Academic press, 1990. Pages 405-428.

\bibitem{Hu1} Hutchings, M. {\it Quantitative embedded contact homology}, arXiv:1005.2260.


\bibitem{Ir} Irie, K. {\it Symplectic capacity and short periodic billiard trajectory}, arXiv:1010.3170.

\bibitem{Ku} K\"unzle, A.F. {\it Une capacit\'e symplectique pour les ensembles convexes et quelques applications}, 
Ph. D. thesis, Universit\'e Paris IX Dauphine, June 1990.

\bibitem{Ku1} K\"unzle, A.F. {\it Singular Hamiltonian systems and symplectic capacities}, Singularities and differential equations, 
171-187,  Banach Center Publ., 33, Polish Acad. Sci., Warsaw, 1996. 

\bibitem{LaMc} Lalonde, F., McDuff, D. {\it The geometry of symplectic
energy,} Ann. of  Math. 141, 349-371 (1995).


\bibitem{MS} McDuff, D., Salamon, D. {\it Introduction to
Symplectic Topology}, 2nd edition, Oxford University Press, Oxford,
England (1998).

\bibitem{MZ} Moser, J., Zehnder, E. {\it Notes on Dynamical Systems},
Courant Lecture Notes in Mathematics, 12. New York University
(2005).

\bibitem{Oh} Oh, Y-G. {\it Chain level Floer theory and Hofer's geometry of the Hamiltonian diffeomorphism
group,} Asian J. Math. 6 (2002), no. 4, 579-624.

\bibitem{R} Rabinowitz, P. H. {\it Periodic solutions of Hamiltonian systems,}
Comm. Pure Appl. Math. 31 (1978), no. 2, 157-184.

\bibitem{Roc}  Rockafellar, R. T. {\it Convex Analysis}, 
Princeton Landmarks in Mathematics. Princeton Paperbacks. Princeton University Press, Princeton, NJ, 1997

\bibitem{Sch} Schneider, R. {\it Convex Bodies: the Brunn-Minkowski theory,}  Encyclopedia of
Mathematics and its Applications, 44. Cambridge University Press, 1993.

\bibitem{Sch1} Schneider, R. {\it Stability for some extremal properties of the simplex}, 
J. Geom. 96 (2009), no. 1-2, 135-148. 

\bibitem{Ste} Steinhagen, P. {\it \"Uber die gr\"oste Kugel in einer konvexen Punktmenge},
Abh. math. Seminar Hamburg Univ., 1 (1992), 15-26.

\bibitem{T} Tabachnikov, S. {\it  Geometry and billiards,}
Student Mathematical Library, 30. American Mathematical Society,
Providence, RI; Mathematics Advanced Study Semesters, University
Park, PA, 2005.

\bibitem{V1} Viterbo, C. {\it Metric and isoperimetic problems in symplectic geometry} , J. Amer. Math. Soc. 13, no.
2, 411-431 (2000).

\bibitem{V2} Viterbo, C. {\it Symplectic topology as the geometry of
generating functions,} Math. Ann. 292, 685-710 (1992).

\bibitem{V3} Viterbo, C. {\it Functors and computations in Floer homology I}, Geom. Funct. Anal, 9, 985-1033, 1999.

\bibitem{W} Weinstein, A. {\it Periodic orbits for convex
Hamiltonian systems}, Ann. Math., 108:507-518, 1978.

\end{thebibliography}
\end{document}